\documentclass{amsart}

\usepackage{hyperref, fullpage, setspace}
\usepackage{amsmath, amsthm, amssymb}

\usepackage{tikz}
\usetikzlibrary{fit}
\usetikzlibrary{decorations.pathmorphing}
\tikzstyle{vertex}=[circle,draw,fill=black,inner sep=1pt]

\title{Forbidden induced subgraphs of double-split graphs}
\author{Boris Alexeev}
\author{Alexandra Fradkin}
\author{Ilhee Kim}

\newtheorem {theorem}            {}[section]

\newcommand{\betweenspace}{\\ \\}
\newcommand{\hsp}{\hspace{0pt}}
\newcommand{\step}[1]{\hsp\newline \emph{#1} \newline}

\newcommand{\co}[1]{\ensuremath{\overline{#1}}}
\newcommand{\iso}{\ensuremath{\cong}}
\renewcommand{\emptyset}{\ensuremath{\varnothing}}

\newcommand{\upright}[1]{\ensuremath{\mathrm{#1}}}
\newcommand{\watch} {\upright{watch}}
\newcommand{\tv}    {\upright{TV}}
\newcommand{\flag}  {\upright{flag}}
\newcommand{\fish}  {\upright{fish}}
\newcommand{\domino}{\upright{domino}}
\newcommand{\tent}  {\upright{tent}}

\newcommand{\figurehere}{\begin{figure}[ht]
\begin{tikzpicture}[scale=1.14]
\node(a0) at (0,0.542705) [vertex] {};
\node(a1) at (0.570634,0.128115) [vertex] {};
\node(a2) at (0.352671,-0.542705) [vertex] {};
\node(a3) at (-0.352671,-0.542705) [vertex] {};
\node(a4) at (-0.570634,0.128115) [vertex] {};
\path(a0) edge (a1)(a1) edge (a2)(a2) edge (a3)(a3) edge (a4)(a4) edge (a0);
\node at (0,-1.2){$F_{1} = C_5$};\node(b0) at (3.2,0.5) [vertex] {};
\node(b1) at (4.2,-0.5) [vertex] {};
\node(b2) at (3.2,-0.5) [vertex] {};
\node(b3) at (3.7,0) [vertex] {};
\node(b4) at (4.2,0.5) [vertex] {};
\path(b0) edge (b2)(b0) edge (b3)(b0) edge (b4)(b1) edge (b2)(b1) edge (b3)(b1) edge (b4);
\node at (3.7,-1.2){$F_{2} = K_{2,3}$};\node(c0) at (6.4,0) [vertex] {};
\node(c1) at (6.9,0) [vertex] {};
\node(c2) at (7.4,0.5) [vertex] {};
\node(c3) at (7.9,0) [vertex] {};
\node(c4) at (7.4,-0.5) [vertex] {};
\node(c5) at (8.4,0) [vertex] {};
\path(c0) edge (c1)(c1) edge (c2)(c1) edge (c4)(c2) edge (c3)(c4) edge (c3)(c3) edge (c5);
\node at (7.4,-1.2){$F_{3} = \watch$};\node(d0) at (10.6,0.7) [vertex] {};
\node(d1) at (11.6,0.7) [vertex] {};
\node(d2) at (11.6,-0.3) [vertex] {};
\node(d3) at (10.6,-0.3) [vertex] {};
\node(d4) at (10.6,-0.7) [vertex] {};
\node(d5) at (11.6,-0.7) [vertex] {};
\path(d0) edge (d1)(d1) edge (d2)(d2) edge (d3)(d3) edge (d0)(d4) edge (d5);
\node at (11.1,-1.2){$F_{4} = \tv$};\node(e0) at (-0.375,-3.75) [vertex] {};
\node(e1) at (-0.375,-3.375) [vertex] {};
\node(e2) at (-0.375,-3) [vertex] {};
\node(e3) at (-0.375,-2.25) [vertex] {};
\node(e4) at (0.375,-2.25) [vertex] {};
\node(e5) at (0.375,-3) [vertex] {};
\path(e0) edge (e1)(e1) edge (e2)(e2) edge (e3)(e3) edge (e4)(e4) edge (e5)(e5) edge (e2);
\node at (0,-4.2){$F_{5} = \flag$};\node(f0) at (3.95,-3) [vertex] {};
\node(f1) at (3.45,-3.5) [vertex] {};
\node(f2) at (2.95,-3) [vertex] {};
\node(f3) at (3.45,-2.5) [vertex] {};
\node(f4) at (4.45,-2.75) [vertex] {};
\node(f5) at (4.45,-3.25) [vertex] {};
\path(f0) edge (f1)(f1) edge (f2)(f2) edge (f3)(f3) edge (f0)(f0) edge (f4)(f4) edge (f5)(f5) edge (f0);
\node at (3.7,-4.2){$F_{6} = \fish$};\node(g0) at (7.4,-2.5) [vertex] {};
\node(g1) at (7.025,-3) [vertex] {};
\node(g2) at (7.775,-3) [vertex] {};
\node(g3) at (6.65,-3.5) [vertex] {};
\node(g4) at (7.4,-3.5) [vertex] {};
\node(g5) at (8.15,-3.5) [vertex] {};
\path(g1) edge (g2)(g3) edge (g4)(g4) edge (g5);
\node at (7.4,-4.2){$F_{7} = P_0\cup P_1\cup P_2$};\node(h0) at (10.6,-3.15) [vertex] {};
\node(h1) at (10.6,-2.45) [vertex] {};
\node(h2) at (11.6,-2.45) [vertex] {};
\node(h3) at (11.6,-3.15) [vertex] {};
\node(h4) at (11.6,-3.55) [vertex] {};
\node(h5) at (10.6,-3.55) [vertex] {};
\path(h0) edge (h1)(h1) edge (h2)(h2) edge (h3)(h4) edge (h5);
\node at (11.1,-4.2){$F_{8} = P_1\cup P_3$};\node(i0) at (0,-6.5625) [vertex] {};
\node(i1) at (0,-6.1875) [vertex] {};
\node(i2) at (0,-5.8125) [vertex] {};
\node(i3) at (-0.5,-5.4375) [vertex] {};
\node(i4) at (0.5,-5.4375) [vertex] {};
\node(i5) at (0.375,-6.1875) [vertex] {};
\path(i0) edge (i1)(i1) edge (i2)(i2) edge (i3)(i3) edge (i4)(i4) edge (i2);
\node at (0,-7.2){$F_{9}$};\node(j0) at (3.7,-6.5625) [vertex] {};
\node(j1) at (3.7,-6.1875) [vertex] {};
\node(j2) at (3.7,-5.8125) [vertex] {};
\node(j3) at (3.2,-5.4375) [vertex] {};
\node(j4) at (4.2,-5.4375) [vertex] {};
\node(j5) at (4.075,-6.1875) [vertex] {};
\path(j0) edge (j1)(j1) edge (j2)(j2) edge (j3)(j3) edge (j4)(j4) edge (j2)(j2) edge (j5);
\node at (3.7,-7.2){$F_{10}$};\node(k0) at (6.9,-6.1875) [vertex] {};
\node(k1) at (6.9,-5.4375) [vertex] {};
\node(k2) at (7.4,-5.8125) [vertex] {};
\node(k3) at (7.9,-6.1875) [vertex] {};
\node(k4) at (7.9,-5.4375) [vertex] {};
\node(k5) at (7.4,-6.5625) [vertex] {};
\path(k0) edge (k1)(k1) edge (k2)(k2) edge (k3)(k3) edge (k4)(k4) edge (k2)(k2) edge (k0);
\node at (7.4,-7.2){$F_{11}$};\node(l0) at (10.6,-6.1875) [vertex] {};
\node(l1) at (10.6,-5.4375) [vertex] {};
\node(l2) at (11.1,-5.8125) [vertex] {};
\node(l3) at (11.6,-6.1875) [vertex] {};
\node(l4) at (11.6,-5.4375) [vertex] {};
\node(l5) at (11.1,-6.5625) [vertex] {};
\path(l0) edge (l1)(l1) edge (l2)(l2) edge (l3)(l3) edge (l4)(l4) edge (l2)(l2) edge (l0)(l2) edge (l5);
\node at (11.1,-7.2){$F_{12}$};\node(m0) at (0,-8.429709) [vertex] {};
\node(m1) at (0.469099,-8.655615) [vertex] {};
\node(m2) at (0.584957,-9.163222) [vertex] {};
\node(m3) at (0.26033,-9.570291) [vertex] {};
\node(m4) at (-0.26033,-9.570291) [vertex] {};
\node(m5) at (-0.584957,-9.163222) [vertex] {};
\node(m6) at (-0.469099,-8.655615) [vertex] {};
\path(m0) edge (m1)(m1) edge (m2)(m2) edge (m3)(m3) edge (m4)(m4) edge (m5)(m5) edge (m6)(m6) edge (m0);
\node at (0,-10.2){$F_{13} = C_7$};\node(n0) at (4,-8.480385) [vertex] {};
\node(n1) at (4.3,-9) [vertex] {};
\node(n2) at (4,-9.519615) [vertex] {};
\node(n3) at (3.4,-9.519615) [vertex] {};
\node(n4) at (3.1,-9) [vertex] {};
\node(n5) at (3.4,-8.480385) [vertex] {};
\node(n6) at (3.7,-9) [vertex] {};
\path(n0) edge (n1)(n1) edge (n2)(n2) edge (n3)(n3) edge (n4)(n4) edge (n5)(n6) edge (n1)(n6) edge (n2)(n6) edge (n3)(n6) edge (n4);
\node at (3.7,-10.2){$F_{14}$};\node(o0) at (7.7,-8.480385) [vertex] {};
\node(o1) at (8,-9) [vertex] {};
\node(o2) at (7.7,-9.519615) [vertex] {};
\node(o3) at (7.1,-9.519615) [vertex] {};
\node(o4) at (6.8,-9) [vertex] {};
\node(o5) at (7.1,-8.480385) [vertex] {};
\node(o6) at (7.4,-9) [vertex] {};
\path(o0) edge (o1)(o1) edge (o2)(o2) edge (o3)(o3) edge (o4)(o4) edge (o5)(o6) edge (o1)(o6) edge (o2)(o6) edge (o3)(o6) edge (o4)(o6) edge (o5);
\node at (7.4,-10.2){$F_{15}$};\node(p0) at (11.4,-8.480385) [vertex] {};
\node(p1) at (11.7,-9) [vertex] {};
\node(p2) at (11.4,-9.519615) [vertex] {};
\node(p3) at (10.8,-9.519615) [vertex] {};
\node(p4) at (10.5,-9) [vertex] {};
\node(p5) at (10.8,-8.480385) [vertex] {};
\node(p6) at (11.1,-9) [vertex] {};
\path(p0) edge (p1)(p1) edge (p2)(p2) edge (p3)(p3) edge (p4)(p4) edge (p5)(p6) edge (p0)(p6) edge (p1)(p6) edge (p2)(p6) edge (p3)(p6) edge (p4)(p6) edge (p5);
\node at (11.1,-10.2){$F_{16}$};\node(q0) at (0.3,-11.480385) [vertex] {};
\node(q1) at (0.6,-12) [vertex] {};
\node(q2) at (0.3,-12.519615) [vertex] {};
\node(q3) at (-0.3,-12.519615) [vertex] {};
\node(q4) at (-0.6,-12) [vertex] {};
\node(q5) at (-0.3,-11.480385) [vertex] {};
\node(q6) at (0,-12) [vertex] {};
\path(q0) edge (q1)(q1) edge (q2)(q2) edge (q3)(q3) edge (q4)(q4) edge (q5)(q5) edge (q0);
\node at (0,-13.2){$F_{17} = C_6\cup P_0$};\node(r0) at (4,-11.480385) [vertex] {};
\node(r1) at (4.3,-12) [vertex] {};
\node(r2) at (4,-12.519615) [vertex] {};
\node(r3) at (3.4,-12.519615) [vertex] {};
\node(r4) at (3.1,-12) [vertex] {};
\node(r5) at (3.4,-11.480385) [vertex] {};
\node(r6) at (3.7,-12) [vertex] {};
\path(r0) edge (r1)(r1) edge (r2)(r2) edge (r3)(r3) edge (r4)(r4) edge (r5)(r5) edge (r0)(r6) edge (r1)(r6) edge (r2)(r6) edge (r3)(r6) edge (r4)(r6) edge (r5);
\node at (3.7,-13.2){$F_{18}$};\node(s0) at (7.7,-11.480385) [vertex] {};
\node(s1) at (8,-12) [vertex] {};
\node(s2) at (7.7,-12.519615) [vertex] {};
\node(s3) at (7.1,-12.519615) [vertex] {};
\node(s4) at (6.8,-12) [vertex] {};
\node(s5) at (7.1,-11.480385) [vertex] {};
\node(s6) at (7.4,-12) [vertex] {};
\path(s0) edge (s1)(s1) edge (s2)(s2) edge (s3)(s3) edge (s4)(s4) edge (s5)(s5) edge (s0)(s6) edge (s0)(s6) edge (s1)(s6) edge (s2)(s6) edge (s3)(s6) edge (s4)(s6) edge (s5);
\node at (7.4,-13.2){$F_{19} = W_6$};\node(t0) at (11.1,-11.429709) [vertex] {};
\node(t1) at (11.569099,-11.655615) [vertex] {};
\node(t2) at (11.684957,-12.163222) [vertex] {};
\node(t3) at (11.36033,-12.570291) [vertex] {};
\node(t4) at (10.83967,-12.570291) [vertex] {};
\node(t5) at (10.515043,-12.163222) [vertex] {};
\node(t6) at (10.630901,-11.655615) [vertex] {};
\path(t0) edge (t2)(t2) edge (t5)(t5) edge (t0)(t0) edge (t1)(t1) edge (t2)(t2) edge (t3)(t3) edge (t4)(t4) edge (t5)(t5) edge (t6)(t6) edge (t0);
\node at (11.1,-13.2){$F_{20}$};\node(u0) at (0.9,-15) [vertex] {};
\node(u1) at (0.3,-15) [vertex] {};
\node(u2) at (0,-15.519615) [vertex] {};
\node(u3) at (-0.6,-15.519615) [vertex] {};
\node(u4) at (-0.9,-15) [vertex] {};
\node(u5) at (0,-14.480385) [vertex] {};
\node(u6) at (-0.6,-14.480385) [vertex] {};
\path(u0) edge (u1)(u1) edge (u2)(u2) edge (u3)(u3) edge (u4)(u4) edge (u5)(u5) edge (u6)(u6) edge (u3)(u3) edge (u5)(u5) edge (u1);
\node at (0,-16.2){$F_{21}$};\node(v0) at (3.7,-14.5) [vertex] {};
\node(v1) at (4.175528,-14.845492) [vertex] {};
\node(v2) at (3.993893,-15.404508) [vertex] {};
\node(v3) at (3.406107,-15.404508) [vertex] {};
\node(v4) at (3.224472,-14.845492) [vertex] {};
\node(v5) at (3.2,-14.2) [vertex] {};
\node(v6) at (4.2,-14.2) [vertex] {};
\node(v7) at (4.2,-15.8) [vertex] {};
\node(v8) at (3.2,-15.8) [vertex] {};
\path(v0) edge (v1)(v0) edge (v2)(v0) edge (v3)(v0) edge (v4)(v1) edge (v2)(v1) edge (v3)(v1) edge (v4)(v2) edge (v3)(v2) edge (v4)(v3) edge (v4)(v0) edge (v5)(v0) edge (v6)(v7) edge (v8)(v7) edge (v1)(v7) edge (v2)(v8) edge (v3)(v8) edge (v4);
\node at (3.7,-16.2){$F_{22}$};\node(w0) at (6.65,-14.25) [vertex] {};
\node(w1) at (7.4,-14.25) [vertex] {};
\node(w2) at (8.15,-14.25) [vertex] {};
\node(w3) at (6.65,-15) [vertex] {};
\node(w4) at (7.4,-15) [vertex] {};
\node(w5) at (8.15,-15) [vertex] {};
\node(w6) at (6.65,-15.75) [vertex] {};
\node(w7) at (7.4,-15.75) [vertex] {};
\node(w8) at (8.15,-15.75) [vertex] {};
\path(w0) edge (w1)(w1) edge (w2)(w3) edge (w4)(w4) edge (w5)(w6) edge (w7)(w7) edge (w8)(w0) edge (w3)(w3) edge (w6)(w1) edge (w4)(w4) edge (w7)(w2) edge (w5)(w5) edge (w8);
\draw(w3) .. controls (7.4,-14.75) and (7.4,-14.75) .. (w5);
\draw(w2) .. controls (7.9,-15) and (7.9,-15) .. (w8);
\draw(w0) .. controls (7.4,-14) and (7.4,-14) .. (w2);
\draw(w6) .. controls (7.4,-15.5) and (7.4,-15.5) .. (w8);
\draw(w1) .. controls (7.15,-15) and (7.15,-15) .. (w7);
\draw(w0) .. controls (6.4,-15) and (6.4,-15) .. (w6);
\node at (7.4,-16.2){$F_{23} = L(K_{3,3})$};
\end{tikzpicture}
\caption{The family $\mathcal{F}$: these 23 graphs, and their
  complements, are the minimal forbidden induced subgraphs for
  double-split graphs.  Only $F_1 = C_5$ and $F_{23} = L(K_{3,3})$ are
  self-complementary.}
\label{fig:double-split}
\end{figure}}

\begin{document}
\begin{abstract}
  In the course of proving the strong perfect graph theorem, Chudnovsky, Robertson, Seymour, and Thomas showed that every perfect graph either belongs to one of five basic classes or admits one of several decompositions.  Four of the basic classes are closed under taking induced subgraphs (and have known forbidden subgraph characterizations), while the fifth one, consisting of double-split graphs, is not.

A graph is doubled if it is an induced subgraph of a double-split graph.  We find the forbidden induced subgraph characterization of doubled graphs; it contains 44 graphs.
\end{abstract}
\maketitle

\section{Introduction}

A key ingredient in the proof of the strong perfect graph theorem by
Chudnovsky, Robertson, Seymour, and Thomas~\cite{spgt} is a
decomposition theorem for all perfect graphs.  This decomposition
theorem states that all perfect graphs either belong to one of five
basic classes or admit one of several decompositions.  The five basic
classes are bipartite graphs, complements of bipartite graphs, line
graphs of bipartite graphs, complements of line graphs of bipartite
graphs, and double-split graphs. The first four classes are closed under
taking induced subgraphs and have known characterizations in terms of
minimal forbidden induced subgraphs.  Indeed, a forbidden
induced subgraph characterization is known for the union of these
four classes~\cite{zver}.  However, double-split graphs are not closed under taking induced subgraphs, and hence do not have such a characterization.

In this paper, we consider the downward closure of double-split graphs
under induced subgraphs (that is, double-split graphs and all of their
induced subgraphs) and we characterize this class in terms of minimal
forbidden induced subgraphs.  Unlike the lists for the other four basic
classes, the one for this class of graphs is finite.

All graphs considered in this paper are finite and have no loops or
multiple edges.  For a graph $G$ we denote its vertex set by $V(G)$ and
its edge set by $E(G)$.  The complement of $G$ is denoted by $\co{G}$.
A \emph{clique} in a graph $G$ is a set of vertices all pairwise
adjacent and a \emph{stable set} is a clique in $\co{G}$.  For $A
\subseteq V(G)$, we denote the subgraph of $G$ induced on $A$ by $G|A$,
sometimes further abbreviating $G|\{u,v,w\}$ by $G|uvw$.  The notation
$G\iso H$ means $G$ is isomorphic to $H$.  For $v \in V(G)$, we denote
the set of neighbors of $v$ in $G$ by $N_G(v)$ and for $X \subseteq
V(G)$, we denote by $N_X(v)$ the set of neighbors of $v$ in $G|X$.

Let $X,Y \subseteq V(G)$ with $X \cap Y=\emptyset$.  We say that $X$ and
$Y$ are \emph{complete} to each other if every vertex of $X$ is adjacent
to every vertex of $Y$, and we say that they are \emph{anticomplete} if
no vertex of $X$ is adjacent to a member of $Y$.  For an integer $i \geq
0$, let $P_i,C_i$ denote the path and cycle with $i$ edges,
respectively.

For integers $a,b \geq 0$, let $M_{a,b}$ be the graph on $2a+b$ vertices
consisting of the disjoint union of $a$ edges and $b$ isolated vertices.
We say that a graph $G$ is \emph{semi-matched} if it is isomorphic to
$M_{a,b}$ for some $a,b \geq 0$ and we say that it is \emph{matched} if
in addition $b=0$.  Similarly, we say that $G$ is
\emph{semi-antimatched} if it is isomorphic to some $\co{M_{a,b}}$ and
\emph{antimatched} if in addition $b=0$.

Let $A,B \subseteq V(G)$ such that $A \cap B = \emptyset$, $A$ is
semi-matched, and $B$ is semi-antimatched.  We say that $A$ and $B$ are
\emph{aligned} if the following holds:
{\figurehere}
\begin{itemize}
\item for all adjacent $u,v \in A$ and all $w \in B$, $w$ is adjacent to
  exactly one of $u$ and $v$
\item for all $u \in A$ and non-adjacent $x,y \in B$, $u$ is adjacent to
  exactly one of $x$ and $y$.
\end{itemize}

A graph $G$ is \emph{split} if its vertex set $V(G)$ can be partitioned
into a clique and a stable set.  A graph $G$ is \emph{double-split} if
its vertex set $V(G)$ can be partitioned into two sets, $A$ and $B$,
such that the following holds:
\begin{itemize}
\item $G|A$ is matched,
\item $G|B$ is antimatched, and
\item $A$ and $B$ are aligned.
\end{itemize}

It is easy to see that every split graph is an induced subgraph of many
double-split graphs.  Also, every induced subgraph of a split graph is
also split.  Split graphs have a well-known forbidden induced subgraphs
characterization:

\begin{theorem}\label{split}[Foldes and Hammer \cite{fh}]
  A graph is split if and only if it does not contain $C_4, \co{C_4}$,
  or $C_5$ as an induced subgraph.
\end{theorem}

In this paper we consider a class of graphs that includes both split and
double-split graphs.  We say a graph $G$ is \emph{doubled} if there
exists a double-split graph $H$ that contains $G$ as an induced
subgraph.  Notice that a graph $G$ is double-split if and only if
$\co{G}$ is double-split, and hence a graph $G$ is doubled if and only
if $\co{G}$ is doubled. The main result of this paper is the following:

\begin{theorem}\label{dsplit}
A graph is doubled if and only if it does not contain any graphs in
$\mathcal{F}$, the family of graphs illustrated in
Figure~\ref{fig:double-split}.
\end{theorem}

It follows that $\mathcal{F}$ is the list of minimal forbidden induced
subgraphs for double-split graphs.  The idea for our proof of \ref{dsplit} is as follows.  To prove the "if" part of
\ref{dsplit}, we assume that $G$ is not split, hence
contains one of $C_4$, $\co{C_4}$, and $C_5$.  Since $C_5$ is in
$\mathcal{F}$ and the class of doubled graphs is self complementary, we
may assume that $G$ has $C_4$ as an induced subgraph.  However, since $C_4$ is a doubled
graph in two different ways (all four vertices can appear on the
anti-matched side or 2 vertices can appear on the matched side and the other 2 vertices on the
semi-antimatched side),  there is no easy procedure to partition the remaining vertices of the graph.  To avoid this obstacle, we introduce another class of graphs that
lies inbetween the class of split graphs and the class of doubled graphs.
In section 2, we find the forbidden induced subgraph characterization for this class and we use this characterization to prove \ref{dsplit} in section 3.

\section{Almost-split graphs}

We say a graph $G$ is {\it almost-split} if $G$ is doubled and there
exists $v \in V(G)$ such that $G | (V(G) \setminus \{v\}$) is split. In
other words, $G$ is almost-split if there is at most one pair matched or
antimatched.  Note that every split graph is almost-split and every
almost-split graph is doubled.  In this section we present the list of
forbidden induced subgraphs for the class of almost-split graphs.

\begin{theorem}\label{asplit}
A graph is almost-split if and only if it does not contain any graphs in
the circus, the list of graphs illustrated in Figure~\ref{fig:asplit} along with their complements.
\end{theorem}
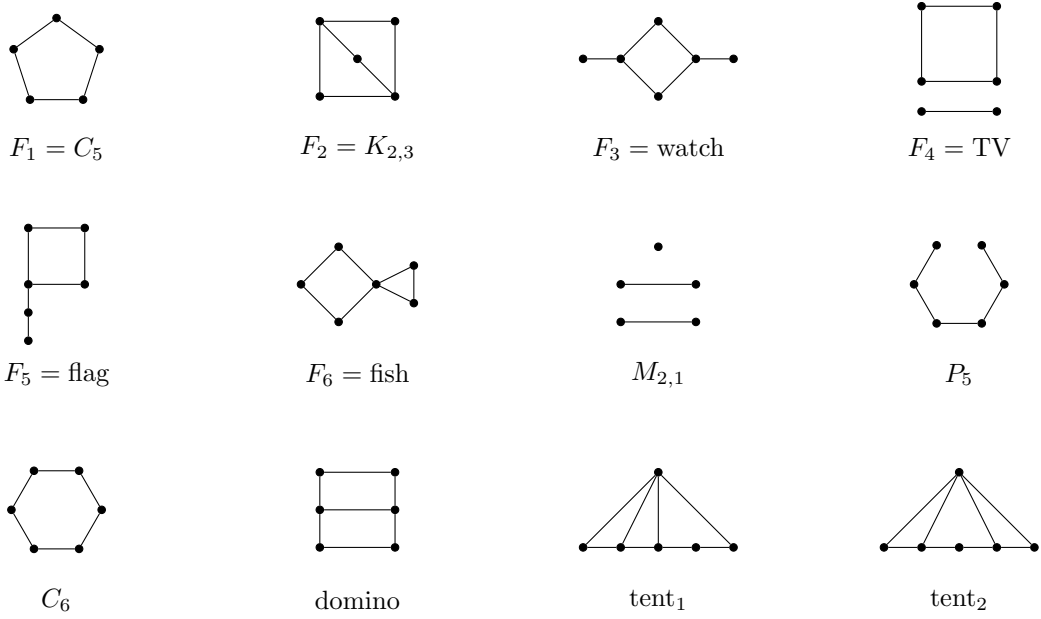
\begin{figure}[ht]
\begin{tikzpicture}
\node(x0) at (0,0.542705) [vertex] {};
\node(x1) at (0.570634,0.128115) [vertex] {};
\node(x2) at (0.352671,-0.542705) [vertex] {};
\node(x3) at (-0.352671,-0.542705) [vertex] {};
\node(x4) at (-0.570634,0.128115) [vertex] {};
\path(x0) edge (x1)(x1) edge (x2)(x2) edge (x3)(x3) edge (x4)(x4) edge (x0);
\node at (0,-1.2){$F_{1} = C_5$};\node(y0) at (3.5,0.5) [vertex] {};
\node(y1) at (4.5,-0.5) [vertex] {};
\node(y2) at (3.5,-0.5) [vertex] {};
\node(y3) at (4,0) [vertex] {};
\node(y4) at (4.5,0.5) [vertex] {};
\path(y0) edge (y2)(y0) edge (y3)(y0) edge (y4)(y1) edge (y2)(y1) edge (y3)(y1) edge (y4);
\node at (4,-1.2){$F_{2} = K_{2,3}$};\node(z0) at (7,0) [vertex] {};
\node(z1) at (7.5,0) [vertex] {};
\node(z2) at (8,0.5) [vertex] {};
\node(z3) at (8.5,0) [vertex] {};
\node(z4) at (8,-0.5) [vertex] {};
\node(z5) at (9,0) [vertex] {};
\path(z0) edge (z1)(z1) edge (z2)(z1) edge (z4)(z2) edge (z3)(z4) edge (z3)(z3) edge (z5);
\node at (8,-1.2){$F_{3} = \text{watch}$};\node(aa0) at (11.5,0.7) [vertex] {};
\node(aa1) at (12.5,0.7) [vertex] {};
\node(aa2) at (12.5,-0.3) [vertex] {};
\node(aa3) at (11.5,-0.3) [vertex] {};
\node(aa4) at (11.5,-0.7) [vertex] {};
\node(aa5) at (12.5,-0.7) [vertex] {};
\path(aa0) edge (aa1)(aa1) edge (aa2)(aa2) edge (aa3)(aa3) edge (aa0)(aa4) edge (aa5);
\node at (12,-1.2){$F_{4} = \text{TV}$};\node(ab0) at (-0.375,-3.75) [vertex] {};
\node(ab1) at (-0.375,-3.375) [vertex] {};
\node(ab2) at (-0.375,-3) [vertex] {};
\node(ab3) at (-0.375,-2.25) [vertex] {};
\node(ab4) at (0.375,-2.25) [vertex] {};
\node(ab5) at (0.375,-3) [vertex] {};
\path(ab0) edge (ab1)(ab1) edge (ab2)(ab2) edge (ab3)(ab3) edge (ab4)(ab4) edge (ab5)(ab5) edge (ab2);
\node at (0,-4.2){$F_{5} = \text{flag}$};\node(ac0) at (4.25,-3) [vertex] {};
\node(ac1) at (3.75,-3.5) [vertex] {};
\node(ac2) at (3.25,-3) [vertex] {};
\node(ac3) at (3.75,-2.5) [vertex] {};
\node(ac4) at (4.75,-2.75) [vertex] {};
\node(ac5) at (4.75,-3.25) [vertex] {};
\path(ac0) edge (ac1)(ac1) edge (ac2)(ac2) edge (ac3)(ac3) edge (ac0)(ac0) edge (ac4)(ac4) edge (ac5)(ac5) edge (ac0);
\node at (4,-4.2){$F_{6} = \text{fish}$};\node(ad0) at (8,-2.5) [vertex] {};
\node(ad1) at (7.5,-3) [vertex] {};
\node(ad2) at (8.5,-3) [vertex] {};
\node(ad3) at (7.5,-3.5) [vertex] {};
\node(ad4) at (8.5,-3.5) [vertex] {};
\path(ad1) edge (ad2)(ad3) edge (ad4);
\node at (8,-4.2){$M_{2,1}$};\node(ae0) at (12.3,-2.480385) [vertex] {};
\node(ae1) at (12.6,-3) [vertex] {};
\node(ae2) at (12.3,-3.519615) [vertex] {};
\node(ae3) at (11.7,-3.519615) [vertex] {};
\node(ae4) at (11.4,-3) [vertex] {};
\node(ae5) at (11.7,-2.480385) [vertex] {};
\path(ae0) edge (ae1)(ae1) edge (ae2)(ae2) edge (ae3)(ae3) edge (ae4)(ae4) edge (ae5);
\node at (12,-4.2){$P_5$};\node(af0) at (0.3,-5.480385) [vertex] {};
\node(af1) at (0.6,-6) [vertex] {};
\node(af2) at (0.3,-6.519615) [vertex] {};
\node(af3) at (-0.3,-6.519615) [vertex] {};
\node(af4) at (-0.6,-6) [vertex] {};
\node(af5) at (-0.3,-5.480385) [vertex] {};
\path(af0) edge (af1)(af1) edge (af2)(af2) edge (af3)(af3) edge (af4)(af4) edge (af5)(af5) edge (af0);
\node at (0,-7.2){$C_6$};\node(ag0) at (3.5,-5.5) [vertex] {};
\node(ag1) at (4.5,-5.5) [vertex] {};
\node(ag2) at (4.5,-6) [vertex] {};
\node(ag3) at (3.5,-6) [vertex] {};
\node(ag4) at (3.5,-6.5) [vertex] {};
\node(ag5) at (4.5,-6.5) [vertex] {};
\path(ag0) edge (ag1)(ag1) edge (ag2)(ag2) edge (ag3)(ag3) edge (ag4)(ag4) edge (ag5)(ag5) edge (ag2)(ag0) edge (ag3);
\node at (4,-7.2){$\text{domino}$};\node(ah0) at (8,-5.5) [vertex] {};
\node(ah1) at (7,-6.5) [vertex] {};
\node(ah2) at (7.5,-6.5) [vertex] {};
\node(ah3) at (8,-6.5) [vertex] {};
\node(ah4) at (8.5,-6.5) [vertex] {};
\node(ah5) at (9,-6.5) [vertex] {};
\path(ah0) edge (ah1)(ah0) edge (ah2)(ah0) edge (ah3)(ah0) edge (ah5)(ah1) edge (ah2)(ah2) edge (ah3)(ah3) edge (ah4)(ah4) edge (ah5);
\node at (8,-7.2){$\text{tent}_1$};\node(ai0) at (12,-5.5) [vertex] {};
\node(ai1) at (11,-6.5) [vertex] {};
\node(ai2) at (11.5,-6.5) [vertex] {};
\node(ai3) at (12,-6.5) [vertex] {};
\node(ai4) at (12.5,-6.5) [vertex] {};
\node(ai5) at (13,-6.5) [vertex] {};
\path(ai0) edge (ai1)(ai0) edge (ai2)(ai0) edge (ai4)(ai0) edge (ai5)(ai1) edge (ai2)(ai2) edge (ai3)(ai3) edge (ai4)(ai4) edge (ai5);
\node at (12,-7.2){$\text{tent}_2$};
\end{tikzpicture}
\caption{The ``circus'': these 12 graphs, and their complements, are the
  minimal forbidden induced subgraphs for almost-split graphs.}
\label{fig:asplit}
\end{figure}
\begin{proof}
The ``only if'' part is clear, as it is easy to check that none of the
graphs in the circus are almost-split.  For the ``if'' part, suppose
that $G$ does not contain any graphs in the circus. By \ref{split}, we
may assume that $G$ contains $C_4$ or $\co{C_4}$ since split graphs are
almost-split. Furthermore, since the statement is self-complementary, we
may assume that $G$ contains $C_4$. Let $a,b,c,d \in V(G)$ be such that
$G|abcd \iso C_4$ and $a$ is adjacent to $b$ and $d$.  Let
$S=\{a,b,c,d\}$.

Since $W_4\iso \co{M_{2,1}}$ is in the circus, it follows that for all
$v \in V(G)$, $v$ is not complete to $S$.  For $0 \leq i \leq 3$, let
$A_i \subseteq V(G)\setminus S$ denote the set of vertices that have $i$
neighbors in $S$.  Our goal is to show that there exist adjacent $x,y
\in S$ such that:
\begin{itemize}
\item $A_0 \cup A_1 \cup A_2 \cup \{x,y\}$ contains only one edge
  (namely $xy$), and
\item $A_3\cup (S \setminus \{x,y\})$ is a clique, and
\item every vertex of $A_3\cup (S \setminus \{x,y\})$ is adjacent to
  exactly one of $x$ and $y$.
\end{itemize}

\step{(1) If $A_2 \neq \emptyset$, then there exist $x,y \in S$ such
  that $A_2$ is complete to $\{x,y\}$.  Moreover, $A_2$ is a stable
  set.}

Let $A_{ab}\subseteq A_2$ be those vertices that are adjacent to $a$ and
$b$, and define $A_{ac}$, $A_{ad}$, $A_{bc}$, $A_{bd}$, $A_{cd}$
similarly.  First suppose that $u \in A_{ac} \cup A_{bd}$; then $G|abcdu
\iso K_{2,3}$.  Hence, both $A_{ac}$ and $A_{bd}$ are empty.  Next
suppose there exists $u \in A_{ab}$ and $v \in A_{bc}$.  Then either
$G|abcduv \iso \tent_2$ or $G|acduv\iso C_5$, depending on the adjacency
between $u$ and $v$.  Therefore, at least one of $A_{ab}$ and $A_{bc}$
is empty, and from symmetry the same is true for the pairs
$\{A_{bc},A_{cd}\}$, $\{A_{cd},A_{ad}\}$, and $\{A_{ab},A_{ad}\}$.  We
claim that at least one of $A_{ab}$ and $A_{cd}$ is empty.  For suppose
$u \in A_{ab}$ and $v \in A_{cd}$.  Then $G|abcduv\iso \co{C_6}$ or
$G|abcduv\iso \co{\domino}$, depending on the adjacency between $u$ and
$v$.  Similarly, at least one of $A_{bc}$ and $A_{ad}$ is empty.  We
conclude that at most one of $A_{ab},A_{ac},A_{ad},A_{bc},A_{bd}$, and
$A_{cd}$ is non-empty.  Finally suppose that $u,v \in A_2$ are adjacent.
Then $G|abcduv\iso \co{\watch}$.  Hence, $A_2$ is a stable set.  This
proves (1).

\step{(2) There exist adjacent $x,y \in S$ such that $N_S(A_1) \subseteq
  \{x,y\}$.  Moreover, if $A_2 \neq \emptyset$, then $N_S(A_1) \subseteq
  N_S(A_2)$.}

Let $A_a \subseteq A_1$ be those vertices that are adjacent to $a$, and
define $A_b$, $A_c$ and $A_d$ similarly.  We show that at least one of
$A_a$ and $A_c$ is empty.  For suppose that $u \in A_a$ and $v \in A_c$.
Then either $G|abcduv\iso \watch$ or $G|abcuv\iso C_5$, depending on the
adjacency between $u$ and $v$.  Similarly, at least one of $A_b$ and
$A_d$ is empty.  This proves the first part of (2).

Next, let $u\in A_1$ and $v\in A_2$.  Suppose that $N_S(A_1)
\not\subseteq N_S(A_2)$.  From symmetry, we may assume that $u \in A_a$
and $v \in A_{bc}$.  But then either $G|abcduv\iso\co{\tent_1}$ or
$G|acduv\iso C_5$, depending on the adjacency between $u$ and $v$.  This
proves (2).

\step{(3) $A_0 \cup A_1 \cup A_2$ is a stable set.}

First, let $u,v \in A_0$ and suppose that they are adjacent.  Then
$G|abcduv\iso \tv$.  Hence, $A_0$ is a stable set.  Next,
suppose $u,v \in A_1$ and suppose that they are adjacent.  If $u,v$ have
a common neighbor in $S$ then $G|abcduv\iso \fish$.  If
$u,v$ have different neighbors in $S$, then by (2) their neighbors are
adjacent and so $G|abcduv\iso \domino$.  This proves that
$A_1$ is a stable set.  Recall that $A_2$ is a stable set by (1).

Now we show that $A_0$, $A_1$, and $A_2$ are pairwise anticomplete to
each other.  Let $u \in A_0$, $v\in A_1$ and suppose that $u$ and $v$
are adjacent. Then $G|abcduv\iso \flag$.  Next, let $u
\in A_0$ and $v \in A_2$ and again suppose that $u$ and $v$ are
adjacent.  Then $G|abcduv\iso \co{\tent_2}$.  Finally,
let $u \in A_1$ and $v \in A_2$ and suppose that they are adjacent.
Then $G|abcduv\iso \tent_1$.  Therefore, we have shown
that $A_0 \cup A_1 \cup A_2$ is stable.  This proves (3).

\step{(4) There exist adjacent $x,y \in S$ such that $A_3$ is complete
  to $x,y$.  Moreover, for all $u \in A_1 \cup A_2$ and $v \in A_3$,
  $N_S(u) \subseteq N_S(v)$.}

Let $A_{abc}\subseteq A_3$ be the set of vertices that are adjacent to
$a,b$ and $c$, and define $A_{abd}$, $A_{acd}$ and $A_{bcd}$ similarly.
We claim that at least one of $A_{abc}$ and $A_{acd}$ is empty.  For
suppose that $u \in A_{abc}$ and $v \in A_{acd}$.  Then either $G|(S
\cup \{u,v\})\iso \co{\tv}$ or $G|acduv\iso W_4$, depending on the
adjacency between $u$ and $v$.  This proves the claim.  By a similar
argument, at least one of $A_{abd}$ and $A_{bcd}$ is empty.  Therefore,
there exist (at least) 2 adjacent vertices of $S$ that are complete to
$A_3$.

Next, let $u \in A_1 \cup A_2$ and $v \in A_3$ and suppose that $N_S(u)
\not\subseteq N_S(v)$.  From symmetry, we may assume that $v \in
A_{abc}$.  If $u \in A_1$, then $u \in A_d$ and so either
$G|abcduv\iso\co{\fish}$ or $G|acduv\iso K_{2,3}$.  So we may assume
that $u \in A_2$.  Again from symmetry, we may assume that $u \in
A_{cd}$.  But then either $G|abcduv\iso\co{\flag}$ or
$\co{P_5}$, depending on the adjacency between $u$ and $v$.  This proves
(4).

\step{(5) $A_3$ is a clique.}

Let $u,v \in A_3$ and suppose that they are not adjacent.  By (4), there
exist adjacent $x,y \in S$ such that $A_3$ is complete to $\{x,y\}$, and
from symmetry we may assume $\{x,y\}=\{a,b\}$.  First suppose that $u,v
\in A_{abc}$.  Then $G|acduv\iso K_{2,3}$.  Therefore, $A_{abc}$ is a
clique, and similarly so is $A_{abd}$.  Next suppose that $u \in
A_{abc}$ and $v \in A_{abd}$. Then $G|abcduv\iso\co{P_5}$.  Hence, $A_3$
is a clique, and this proves (5).
\betweenspace
From (1), (2), and (4), it follows that there exist adjacent $x,y \in S$
such that $A_3 \cup A_2$ is complete to $\{x,y\}$ and $N_S(A_1)
\subseteq \{x,y\}$.  From symmetry, we may assume that
$\{x,y\}=\{a,b\}$.  Hence, $A_0 \cup A_1 \cup A_2$ is anticomplete to
$\{c,d\}$.  Therefore, by (3), $A_0 \cup A_1 \cup A_2 \cup \{c,d\}$
contains exactly one edge (namely $cd$).  By (4) and (5), $A_3 \cup
\{a,b\}$ is a clique.  Also, since every member of $A_3$ is adjacent to
exactly 3 members of $S$, it follows that for all $u \in A_3 \cup
\{a,b\}$, $u$ is adjacent to exactly one of $c,d$.  Hence, we have shown
that $G$ is almost-split and this proves \ref{asplit}.
\end{proof}

\section{Excluding 6 graphs}

In the previous section, we have seen the 12 minimal forbidden induced
subgraphs (up to taking complements) for almost-split graphs.  Six of
them are doubled and the other six are not.  In this section, we prove
that if a graph contains one of these six doubled graphs but no graphs in
$\mathcal{F}$, then it is doubled.

\begin{theorem}\label{M_{2,1}}
A graph containing $M_{2,1}$ but no graphs in $\mathcal{F}$ is doubled.
\end{theorem}
\begin{proof}
Let $G$ be a graph containing $M_{2,1}$ but no graphs in $\mathcal{F}$.
Let $G|abcde\iso M_{2,1}$, where $bc$ and $de$ are the
two edges; let $S = \{a,b,c,d,e\}$.  For $0\leq i \leq 4$, let $A_i
\subseteq V(G) \setminus S$ denote the set of vertices that have $i$
neighbors in $\{b,c,d,e\}$. Our goal is to show the following:
\begin{itemize}
\item $A_1 = A_3 = A_4 = \emptyset$, and
\item $G|(A_0 \cup S)$ is semi-matched, and
\item $G|A_2$ is semi-antimatched, and
\item $A_0 \cup S$ and $A_2$ are aligned.
\end{itemize}
Together, these statements imply that $G$ is doubled.

\step{(1) $A_1 = A_3 = A_4 = \emptyset$. Also, if $v \in A_2$, then $v$
  is adjacent to exactly one of $b$ and $c$, and to exactly one of $d$
  and $e$.}

If $v \in A_1$, then $G|abcdev\iso F_7$ or $F_8$,
depending on the adjacency between $v$ and $a$. Therefore $A_1$ is
empty.  If $v \in A_3$, then $G|abcdev\iso F_9$ or
$F_{10}$, depending on the adjacency between $v$ and $a$. Therefore
$A_3$ is empty.  And if $v \in A_4$, then $G|abcdev\iso F_{11}$ or
$G|abcdev\iso F_{12}$, depending on the adjacency between $v$ and
$a$. Therefore $A_4$ is empty.

Next, let $v \in A_2$. If $v$ is adjacent to $b$ and $c$, then
$G|bcdev\iso \co{K_{2,3}}$. By symmetry, $v$ is not adjacent to
both of $d$ and $e$.  Hence, $v$ is adjacent to exactly one of $b$ and
$c$ and to exactly one of $d$ and $e$.  This proves (1).

\step{(2) $G|(A_0 \cup S)$ is semi-matched.}

First, we claim that at most one vertex $x \in A_0$ is adjacent to $a$,
and if such a vertex $x$ exists, then $x$ is not adjacent to any other
vertices in $A_0$.  For suppose there are two vertices $x,y \in A_0$,
both adjacent to $a$.  If $x$ and $y$ are adjacent, then $G|abcxy\iso
\co{K_{2,3}}$, and if they are not adjacent, $G|abcdxy\iso F_7$.  So
there is at most one vertex in $A_0$ adjacent to $a$.  Moreover if there
is a vertex $x \in A_0$ adjacent to $a$, $x$ is not adjacent to any
other vertex $y \in A_0$ since otherwise $G|abcdxy\iso F_7$. This proves
the claim.

To prove (2), it is enough to show that there do not exist vertices
$u,v,w \in A_0 \cup \{a,b,c,d,e\}$ such that $G|uvw\iso C_3$ or
$G|uvw\iso P_2$.  If at least one of $u,v,w$ is a member of $S$, then
$G|uvw$ cannot be isomorphic to $C_3$ nor $P_2$ by the claim.  So we may
assume $u,v,w \in A_0$. But now if $G|uvw\iso C_3$, then
$G|bcuvw\iso \co{K_{2,3}}$ and if $G|uvw\iso P_2$, then
$G|bcduvw\iso F_7$.  This proves (2).

\step{(3) Let $u,v \in A_2$ be non-adjacent. Then $N_{\{b,c,d,e\}} (u)$
  is disjoint from $N_{\{b,c,d,e\}} (v)$. Moreover, exactly one of $u$
  and $v$ is adjacent to $a$.}

From (1) and by symmetry, we may assume that $u$ is adjacent to $b$ and
$d$.  Suppose that $v$ is also adjacent to $b$ and $d$. Then
$G|bcdeuv\iso\watch$.  Next, suppose that $v$ is adjacent to $b$ and $e$
(or $c$ and $d$).  Then $G|bdeuv$ (or $G|bcduv$) is isomorphic to $C_5$.
Consequently, $v$ is adjacent to $c$ and $e$.

Moreover, if $u$ and $v$ are both adjacent to $a$, then $G|abcuv\iso
C_5$ and if $u$ and $v$ are both non-adjacent to $a$, then
$G|abcdeuv\iso F_{17}$.  Hence, exactly one of $u$ and $v$ is adjacent
to $a$.  This proves (3).

\step{(4) $G|A_2$ is semi-antimatched.}

It follows easily from (3) that there is no stable set of size 3 in
$A_2$.  Therefore, it is enough to show that there do not exist vertices
$u,v,w \in A_2$ such that $G|uvw$ contains exactly one edge (say
$uv$). For contradiction, suppose that such $u,v,w$ exist.  From (3) and
by symmetry, we may assume that $\{u,v\}$ is complete to $\{b,d\}$, $w$
is complete to $\{c,e\}$, and $N_{\{u,v,w\}} (a)$ is either $\{u,v\}$ or
$\{w\}$.  In the first case, $G|acuvw\iso \co{K_{2,3}}$
and in the second case, $G|abuvw\iso \co{K}_{2,3}$.  Therefore,
there do not exist $u,v,w \in A_2$ such that $G|uvw$ contains exactly
one edge, and this proves (4).
\betweenspace
It remains to show that $G|(A_0 \cup S)$ and $G|A_2$ are aligned.  In
(3), we have shown that for all non-adjacent $u,v \in A_2$ and all $w
\in A_0 \cup S$, $w$ is adjacent to exactly one of $u$ and $v$.  Hence,
it suffices to show that for all $u \in A_2$ and all adjacent $v,w \in
A_0 \cup S$, $u$ is adjacent to exactly one of $v,w$.  So suppose that
for some $u,v,w$ as above, $u$ is adjacent to both of $v,w$.  Let $x,y
\in A_0 \cup S$ be adjacent such that $\{x,y\}$ is disjoint from
$\{v,w\}$ (such $x,y$ exist since $A_0\cup S$ contain at least two
edges).  Then $G|uvwxy\iso \co{K}_{2,3}$.  Next, suppose
that for some $u,v,w$ as above, $u$ is non-adjacent to both of $v,w$.
Note that by (1), $\{v,w\}$ is disjoint from $\{b,c,d,e\}$.  By (1) and
without loss of generality, we may assume that $u$ is adjacent to $b$
and $d$.  But then $G|bceuvw\iso F_7$.  Therefore $G$ is
doubled and this proves \ref{M_{2,1}}. \end{proof}

\begin{theorem}\label{P_5}
A graph containing $P_5$ but no graphs in $\mathcal{F}$ is doubled.
\end{theorem}
\begin{proof}
Let $G$ be a graph containing $P_5$ but no graphs in $\mathcal{F}$.  Let
$G|abcdef\iso P_5$ where $ab$, $bc$, $cd$, $de$, and $ef$ are the five
edges.  By \ref{M_{2,1}}, we may assume that $G$ or $\co{G}$ does not
contain $M_{2,1}$.  Let $S = \{a,b,c,d,e,f\}$.  For $0\leq i \leq 4$,
let $A_i \subseteq V(G) \setminus S$ denote the set of vertices that
have $i$ neighbors in $\{b,c,d,e\}$. Our goal is to show the following:
\begin{itemize}
\item $A_0 = A_2 = A_4 = \emptyset$, and
\item $G|(A_1 \cup \{a,c,d,f\})$ is semi-matched, and
\item $G|(A_3 \cup \{b,e\})$ is semi-antimatched, and
\item $G|(A_1 \cup \{a,c,d,f\})$ and $G|(A_3 \cup \{b,e\})$ are aligned.
\end{itemize}
Together, these statements imply that $G$ is doubled.

\step{(1) $A_0 = A_2 = A_4 = \emptyset$.}

First suppose $v \in A_0$. If $v$ is non-adjacent to $a$, then $G|abdev\iso
M_{2,1}$.  So we may assume that $v$ is adjacent to $a$ and similarly
$v$ is adjacent to $f$; but then $G|abcdefv\iso F_{13}$.
Therefore $A_0$ is empty.

Next, suppoose $v \in A_4$.  If $v$ is not adjacent to both of $a$ and
$f$, then $G|abcdefv\iso F_{14}$.  If $v$ is adjacent to exactly one of
$a$ and $f$, then $G|abcdefv\iso F_{15}$.  So we may assume that $v$ is
adjacent to both $a$ and $f$; but then $G|abcdefv\iso F_{16}$.
Therefore $A_4$ is empty.

Finally suppose $v \in A_2$.  Let $A_{bc}\subseteq A_2$ be those
vertices that are adjacent to $b$ and $c$, and define $A_{bd}$,
$A_{be}$, $A_{cd}$, $A_{ce}$, $A_{de}$ similarly.  Suppose $v \in
A_{bc}$.  If $v$ is adjacent to $f$, then $G|cdefv\iso C_5$; otherwise,
$G|bcefv\iso \co{K_{2,3}}$.  Therefore $A_{bc}$ is empty and by
symmetry, $A_{de}$ is empty.  Next, suppose $v \in A_{bd}$.  If $v$ is
not adjacent to $a$, then $G|abcdev\iso \watch$ and if $v$ is not
adjacent to $f$, then $G|bcdefv\iso \flag$.  So we may assume that $v$
is adjacent to both of $a$ and $f$; but then $G|abdefv\iso \fish$.
Therefore $A_{bd}$ is empty and by symmetry, $A_{ce}$ is empty.  Next,
suppose $v \in A_{be}$. Then $G|bcdev\iso C_5$.  Therefore $A_{be}$ is
empty.  So we may assume that $v \in A_{cd}$.  If $v$ is adjacent to
both of $a$ and $f$, then $G|abcefv\iso \flag$.  If $v$ is adjacent to
$a$ and not to $f$, then $G|abcefv\iso \tv$, and if $v$ is adjacent to
$f$ and not to $a$, then $G|abdefv\iso\tv$.  So we may assume that $v$
is not adjacent to either $a$ or $f$; but then $G|abefv\iso M_{2,1}$.
Therefore $A_2 = \emptyset$ and this proves (1).

\step{(2) If $v \in A_1$, then $v$ is adjacent to either $b$ or $e$ and
  is not adjacent to both $a$ and $f$. Moreover, $A_1$ is a stable set.}

Let $A_{b}\subseteq A_1$ be those vertices that are adjacent to $b$, and
define $A_{c}$, $A_{d}$, $A_{e}$ similarly.  Suppose $v \in A_c$.  If
$v$ is adjacent to $a$, then $G|abcdev\iso \flag$; otherwise,
$G|abdev\iso M_{2,1}$.  Therefore $A_c$ is empty, and by symmetry, $A_d$
is empty.  Suppose $v \in A_b$.  If $v$ is adjacent to $a$, then
$G|abdev\iso \co{K_{2,3}}$ and if $v$ is adjacent to $f$, then
$G|acdfv\iso M_{2,1}$. Therefore $v$ is anticomplete to $\{a,f\}$, and
similarly, every vertex in $A_e$ is anticomplete to $\{a,f\}$.

Next, suppose that $u,v \in A_b$ are adjacent.  Then
$G|bdeuv\iso\co{K_{2,3}}$.  Therefore $A_b$ is a stable set and
similarly, so is $A_e$.  Finally, suppose that $u \in A_b$ and $v \in
A_e$ are adjacent.  Then $G|acduv\iso M_{2,1}$.  Therefore $A_b \cup A_e
= A_1$ is a stable set, and this proves (2).

\step{(3) If $v \in A_3$, then $N_{\{b,c,d,e\}}(v)$ is either
  $\{b,c,e\}$ or $\{b,d,e\}$.  Moreover, if $u,v \in A_3$ are not adjacent,
  then $N_{\{u,v\}} (c) \neq N_{\{u,v\}} (d)$ and $\lvert N_{\{u,v\}}
  (a)\rvert = \lvert N_{\{u,v\}} (f)\rvert = 1$.}

Let $A_{bcd}\subseteq A_3$ be those vertices that are adjacent to $b,c$,
and $d$, and define $A_{bce}$, $A_{bde}$, $A_{cde}$ similarly.  Suppose
$v \in A_{cde}$.  If $v$ is not adjacent to $a$, then
$G|abdev\iso\co{K_{2,3}}$, and if $v$ is adjacent to $a$ and $f$, then
$G|abcefv\iso\fish$.  So we may assume that $v$ is adjacent to $a$ but
not to $f$; but then $G|abcefv\iso \flag$.  Therefore $A_{cde}$ is empty
and similarly, so is $A_{bcd}$.

Now suppose $u,v \in A_{bce}$ are not adjacent; then
$G|cdeuv\iso K_{2,3}$.  Therefore $A_{bce}$ is a clique and similarly,
so is $A_{bde}$.  Suppose $u \in A_{bce}$ and $v \in A_{bde}$ are not
adjacent.  If $u,v$ are both adjacent to $a$, then
$G|abdeuv\iso\co{\flag}$, and if $u,v$ are both non-adjacent to $a$,
then $G|abcdeuv\iso \co{F_{18}}$.  Therefore exactly one of $u$ and $v$
is adjacent to $a$ and similarly, exactly one of $u$ and $v$ is adjacent
to $f$. This proves (3).

\step{(4) $G|(A_3 \cup \{b,e\})$ is semi-antimatched.}

It is enough to show that no set of three vertices $\{u,v,w\} \subseteq
A_3 \cup \{b,e\}$ contains fewer than two edges.  By (3), it is obvious that
there are no stable sets of size 3 in $G|(A_3 \cup \{b,e\})$.  Suppose
$\{u,v,w\}$ contains exactly one edge $uv$.  From (3) and
symmetry, we may assume $u,v \in A_{bce}$ and $w \in A_{bde}$.  But then
$G|cdeuvw\iso \co{\flag}$.  This proves (4).
\betweenspace
From (2) and (4), we have a candidate of a partition for $G$ to be
doubled.  The subgraph $G| (A_1 \cup \{a,c,d,f\})$ contains only one
edge (namely $cd$) and $G|(A_3 \cup \{b,e\})$ is semi-antimatched.
Every $v \in A_3 \cup \{b,e\}$ has exactly one neighbor in $\{c,d\}$ and
from (3), for every non adjacent pair $u,v \in A_3 \cup \{b,e\}$,
$N_{\{u,v\}} (c) \neq N_{\{u,v\}} (d)$.  Also by (3), if $u,v \in A_3
\cup \{b,e\}$ are nonadjacent, $\lvert N_{\{u,v\}} (a)\rvert = \lvert
N_{\{u,v\}} (f)\rvert = 1$.  Moreover, for $w \in A_1$, either $G|abcdew
\iso P_5$ or $G|bcdefw\iso P_5$, and so $\lvert N_{\{u,v\}} (w)\rvert =
1$ for every $w \in A_1$, by an analogous argument to the one above.
Therefore, $G|(A_3 \cup \{b,e\})$ and $G| (A_1 \cup \{a,c,d,f\})$ are
aligned and so $G$ is doubled; this proves \ref{P_5}.
\end{proof}

\begin{theorem}\label{C_6}
A graph containing $\co{C_6}$ but no graphs in $\mathcal{F}$ is doubled.
\end{theorem}
\begin{proof}
Let $G$ be a graph containing $\co{C_6}$ but no graphs in $\mathcal{F}$.
Let $G|abcdef\iso\co{C_6}$ where $\{a,c,e\}$ and $\{b,d,f\}$ are the two
triangles and the remaining edges are $ad, be$, and $cf$. Let $S =
\{a,b,c,d,e,f\}$.  By \ref{M_{2,1}}, we may assume $G$ or $\co{G}$ does
not contain $M_{2,1}$.  For $0\leq i \leq 6$, let $A_i \subseteq V(G)
\setminus S$ denote the set of vertices that have $i$ neighbors in
$S$. Our goal is to show that $A_i = \emptyset$ unless $i = 2$ and 4
vertices of $S$ induce antimatching side and the rest of vertices (two
in $S$ together with vertices in $A_2$) induce matching side so that $G$
is doubled.

\step{(1) $A_i = \emptyset$ for $i = 0,1,3,4,5,6$.}

If $v \in A_0$, then $G|(S \cup \{v\})\iso\co{F_{19}}$, so $A_0$ is
empty.  Also, if $v \in A_6$, then $G|(S \cup \{v\})\iso \co{F_{17}}$,
and so $A_6$ is empty.

Next, suppose $v \in A_1$.  From symmetry, we may assume
$N_S(v) = \{a\}$.  Then $G|(S \cup \{v\})\iso F_{18}$ and therefore $A_1$
is empty.

Next, suppose $v \in A_3$.  From symmetry, we may assume $N_S(v)$ is one
of $\{a,b,c\}, \{a,b,d\}, \{a,c,e\}$.  If $N_S(v) = \{a,b,c\}$, then
$G|abdev\iso K_{2,3}$ and if $N_S(v) = \{a,b,d\}$, then $G|abcfv\iso
C_5$.  So we may assume that $N_S(v) = \{a,c,e\}$; but then
$G|abcdev\iso \co{\watch}$ and so $A_3$ is empty.

Next, suppose $v \in A_4$.  From symmetry, we may
assume $N_S(v)$ is one of $\{a,b,c,d\}, \{a,b,c,e\}$, and $\{a,b,d,e\}$.
If $N_S(v) = \{a,b,c,d\}$, then $G|bcefv\iso K_{2,3}$ and if $N_S(v) =
\{a,b,c,e\}$, then $G|acdefv\iso \co{\watch}$.  So we may assume that $N_S(v) =
\{a,b,d,e\}$; but then $G|abcfv\iso C_5$, and so $A_4$ is empty.

Finally, suppose $v \in A_5$.  From symmetry, we may assume $N_S(v) =
\{b,c,d,e,f\}$.  Then $G|bcefv\iso\co{M_{2,1}}$.  Therefore $A_5$ is
empty and this proves (1).
\betweenspace
For $u,v \in S$, let $A_{uv}\subseteq A_2$ be those vertices that are
adjacent to $u$ and $v$.

\step{(2) $A_{ab} = A_{bc} = A_{cd} = A_{de} = A_{ef} = A_{fa} =
  \emptyset$.}

Suppose $v \in A_{ab}$. Then $G|abcfv\iso C_5$.  Therefore $A_{ab}$ is
empty and similarly, so are $A_{bc}$, $A_{cd}$, $A_{de}$, $A_{ef}$ ,and
$A_{fa}$.

\step{(3) For every $x,y \in S$, $A_{xy}$ is a stable set.}

Suppose $u,v \in A_{ac}$ are adjacent. Then $G|acdfuv\iso\co{\watch}$.
Therefore $A_{ac}$ is a stable set and similarly so are $A_{ae}$,
$A_{ce}$, $A_{bd}$, $A_{bf}$, and $A_{df}$.  Suppose $u,v \in A_{ad}$
are adjacent. Then $G|acdfuv\iso \co{\watch}$.  Therefore $A_{ad}$ is a
stable set and similarly so are $A_{be}$ and $A_{cf}$. This proves
(3).

\step{(4) If $A_{ac}\neq \emptyset$, then $A_{ae} = A_{ce} = A_{bd} =
  A_{bf} = \emptyset$.}

Suppose $u \in A_{ac}$, and $v \in A_{ae}$.  If $u$ and $v$ are
adjacent, then $G|aceuv\iso \co{M_{2,1}}$, and otherwise $G|bcduv\iso
M_{2,1}$.  Therefore if $A_{ac}$ is not empty, then $A_{ae} = \emptyset$
and similarly, $A_{ce} = \emptyset$.

Now suppose $v \in A_{bd}$.  If $u$ and $v$ are adjacent, then
$G|cdfuv\iso C_5$, and otherwise $G|bcefuv\iso\watch$.  Therefore
if $A_{ac}$ is not empty, then $A_{bd} = \emptyset$ and similarly,
$A_{bf} = \emptyset$.  This proves (4).

\step{(5) If $A_{ad}\neq \emptyset$, then $A_{ce} = A_{bf} = \emptyset$.}

Suppose $u \in A_{ad}$ and $v \in A_{ce}$.  If $u$ and $v$ are adjacent,
then $G|cdfuv\iso C_5$, and otherwise $G|cdeuv\iso\co{K_{2,3}}$.
Therefore if $A_{ad}$ is not empty, then $A_{ce}$ is empty and
similarly, $A_{bf}$ is empty as well.  This proves (5).

\step{(6) If $u \in A_{ad}$, then $N_G(u) \setminus S \subseteq
  A_{be} \cup A_{cf}$.}

Suppose $u \in A_{ad}$. Then from (5), $A_{ce} = A_{bf} = \emptyset$,
and from (3), $u$ has no neighbors in $A_{ad}$.  Now suppose $v \in
A_{ac}$ is adjacent to $u$.  Then $G|cdfuv\iso C_5$.  Therefore $u$ is
anticomplete to $A_{ac}$, and similarly, $u$ is anticomplete to
$A_{ae}$, $A_{bd}$, and $A_{df}$ as well.  Therefore $N_G(u) \setminus S
\subseteq A_{be} \cup A_{cf}$, and this proves (6).

\step{(7) If there are adjacent vertices $u \in A_{ad}$ and $v \in
  A_{be} \cup A_{cf}$, then $V(G) = S \cup \{u,v\}$ and $G$ is doubled.}

From symmetry, we may assume $v \in A_{cf}$ is adjacent to $u \in
A_{ad}$.  We know that $A_{ae} \cup A_{ce} \cup A_{bd} \cup A_{bf} =
\emptyset$ by (5).  Suppose $w \in A_{ac} \cup A_{df}$.  Then from (6),
$\{u,v\}$ is anticomplete to $w$ and so $G|beuvw\iso M_{2,1}$.
Therefore $A_{ac} \cup A_{ae} \cup A_{ce} \cup A_{bd} \cup A_{bf} \cup
A_{df} = \emptyset$.

Next, suppose $w (\neq v) \in A_{cf}$.  From (3), $w$ is not adjacent to
$v$. If $w$ is adjacent to $u$, then $G|abfuvw\iso\watch$, and otherwise
$G|beuvw\iso M_{2,1}$.  Therefore $A_{cf} = \{v\}$ and similarly,
$A_{ad} = \{u\}$.

Now suppose $w \in A_{be}$.  If $w$ is anticomplete to $\{u,v\}$, then
$G|beuvw\iso \co{K_{2,3}}$.  Therefore $w$ is adjacent to at least one
of $\{u,v\}$ and by the same logic as above, $A_{be} = \{w\}$.  If $w$
is adjacent to exactly one of $u$ and $v$ (say $u$), then $G|abcuvw\iso
\flag$.  So we may assume that $w$ is adjacent to both $u$ and $v$; but
then $G|(S \cup \{u,v,w\})\iso F_{23}$.  Therefore $A_{be}=\emptyset$.
But then $V(G) = S \cup \{u,v\}$.  Since $G|uvbe$ is matched, $G|abcf$
is antimatched, and the two subgraphs are aligned, it follows that $G$
is doubled.  This proves (7).

\step{(8) If $v \in A_{ac}$, then $N_G(v) \setminus S \subseteq
  A_{df}$.}

Suppose $v \in A_{ac}$. From (3), $v$ has no neighbors in $A_{ac}$.
From (4), $A_{ae} = A_{ce} = A_{bd} = A_{bf} = \emptyset$ and from (5),
$A_{be} = \emptyset$.  Finally, from (6), $v$ is anticomplete to $A_{ad}
\cup A_{cf}$.  Therefore $N_G(v) \setminus S \subseteq A_{df}$, and this
proves (8).

\step{(9) If there are adjacent vertices $u \in A_{ac}$ and $v \in
  A_{df}$, then $V(G) = S \cup \{u,v\}$ and $G$ is doubled.}

From (4) and (5), $A_{ae} = A_{ce} = A_{bd} = A_{bf} = A_{be} =
\emptyset$.  If $w \in A_{ad} \cup A_{cf}$, then from (8), $w$ is
anticomplete to $\{u,v\}$ and so $G|beuvw\iso M_{2,1}$.
Therefore $A_{ad} \cup A_{cf}$ is empty and $V(G) \setminus S = A_{ac}
\cup A_{df}$.

Now suppose $w (\neq v) \in A_{df}$.  From (3), $w$ is not adjacent to
$v$.  If $w$ is adjacent to $u$, then $G|cdeuvw\iso \flag$, and
otherwise $G|beuvw\iso M_{2,1}$.  Therefore $A_{df} = \{v\}$ and
similarly, $A_{ac} = \{u\}$.  Hence, $V(G) = S \cup \{u,v\}$.  Since
$G|uvbe$ is matched, $G|abcf$ is antimatched, and the two subgraphs are
aligned, it follows that $G$ is doubled.  This proves (9).

\step{(10) If $G| (V(G) \setminus S)$ is a stable set, then $G$ is
  doubled.}

Suppose $G| (V(G) \setminus S)$ is a stable set.  First, suppose $A_{ac}
\cup A_{ce} \cup A_{ae} \cup A_{bd} \cup A_{bf} \cup A_{df} \neq
\emptyset$.  From symmetry, we may assume $A_{ac} \neq \emptyset$.  Then
from (4), $A_{ae} = A_{ce} = A_{bd} = A_{bf} = \emptyset$ and from
(5), $A_{be} = \emptyset$.  Therefore every vertex in $V(G) \setminus S$
has exactly one neighbor in $\{a,f\}$ and exactly one neighbor in
$\{c,d\}$. Now it is easy to see that $G$ is doubled with
$G|acdf$ as the antimatched part.

Therefore we may assume $A_{ac} \cup A_{ce} \cup A_{ae} \cup A_{bd} \cup
A_{bf} \cup A_{df} = \emptyset$.  Suppose all three of the sets $A_{ad},
A_{be}$, and $A_{cf}$ are not empty.  Then for $u \in A_{ad}, v \in
A_{be}$, and $w \in A_{cf}$, $G|afuvw\iso M_{2,1}$.
Therefore from symmetry, we may assume $A_{be}$ is empty.  Now again,
every vertex in $V(G) \setminus S$ has exactly one neighbor in $\{a,f\}$ and
exactly one neighbor in $\{c,d\}$, so $G$ is doubled with
$G|acdf$ as the antimatched part.  This proves (10).
\betweenspace
By (10), we may assume that $G|V(G) \setminus S$ contains an edge $uv$.
From symmetry, we may assume $u \in A_{ad}$ or $u \in A_{ac}$.  If $u
\in A_{ad}$, then by (6) and (7), $v \in A_{be} \cup A_{cf}$ and $G$ is
doubled.  So we may assume that $u \in A_{ac}$; but then by (8) and (9),
$v \in A_{df}$ and $G$ is doubled. This proves $\ref{C_6}$.
\end{proof}

\begin{theorem}\label{Domino}
A graph containing $\co{\domino}$ but no graphs in $\mathcal{F}$ is
doubled.
\end{theorem}
\begin{proof}
Let $G$ be a graph containing $\co{\domino}$ but no graphs in
$\mathcal{F}$.  By \ref{M_{2,1}}, \ref{P_5}, and \ref{C_6}, we may
assume that $G$ does not contain $M_{2,1}$, $P_5$, $C_6$, or their
complements as induced subgraphs.  Let $G|abcdef\iso \co{\domino}$,
where $ab$, $bc$, $ca$, $bd$, $ce$, $de$, $df$, and $ef$ are the edges;
let $S = \{a,b,c,d,e,f\}$.  For $0\leq i \leq 4$, let $A_i \subseteq
V(G) \setminus S$ denote the set of vertices that have $i$ neighbors in
$\{b,c,d,e\}$. Our goal is to show the following:
\begin{itemize}
\item $A_0 = A_1 = A_3 = A_4 = \emptyset$, and
\item $G|(A_2 \cup \{a,f\})$ is a stable set, and
\item $G|bcde$ is antimatched, and
\item $A_2 \cup \{a,f\}$ and $\{b,c,d,e\}$ are aligned.
\end{itemize}
Together, these statements imply that $G$ is doubled.

\step{(1) $A_0 = A_1 = A_3 = A_4 = \emptyset$.}

Suppose $v \in A_0$.  If $v$ is complete to $\{a,f\}$, then $G|abdfv\iso
C_5$, and if $v$ is anticomplete to $\{a,f\}$, then $G|acdfv\iso
M_{2,1}$.  So we may assume that $v$ is adjacent to exactly one of $a$
and $f$, say $a$; but then $G|adefv\iso \co{K_{2,3}}$.  Therefore $A_0 =
\emptyset$.

Next, suppose $v \in A_1$.  From symmetry, we may assume
$N_{\{b,c,d,e\}}(v) = \{b\}$.  If $v$ is complete to $\{a,f\}$, then
$G|acefv\iso C_5$, and if $v$ is anticomplete to $\{a,f\}$, then
$G|acdfv\iso M_{2,1}$.  Furthermore, if $v$ is adjacent to $a$ but not
to $f$, then $G|abefv\iso \co{K_{2,3}}$.  So we may assume that $v$ is
adjacent to $f$ but not to $a$; but then $G|bcefv\iso C_5$.  Therefore
$A_1 = \emptyset$.

Next, suppose $v \in A_3$.  From symmetry, we may assume $N_{\{b,c,d,e\}}(v) =
\{b,c,e\}$.  If $v$ is not adjacent to $f$, then
$G|bcdefv\iso\co{\flag}$, and if $v$ is complete to $\{a,f\}$, then
$G|abcdfv\iso \co{\watch}$.  So we may assume that $v$ is adjacent to $f$ but
not to $a$; but then $G|abdefv\iso \co{\fish}$.  Therefore $A_3 = \emptyset.$

Finally, suppose $v \in A_4$. Then $G|bcdev\iso \co{M_{2,1}}$.
Therefore $A_4 = \emptyset.$ This proves (1).
\betweenspace
For $u,v \in \{b,c,d,e\}$, let $A_{uv}\subseteq A_2$ be those vertices
that are adjacent to $u$ and $v$.

\step{(2) $A_{be} = A_{cd} = \emptyset$. Moreover, $A_2 \cup \{a,f\}$ is a
  stable set.}

Suppose $v \in A_{be} \cup A_{cd}$; then $G|bcdev\iso K_{2,3}$.
Therefore $A_{be} = A_{cd} = \emptyset$.  Next, suppose $v \in A_{bc}$. If $v$
is adjacent to $a$, then $G|abcdev\iso \co{\watch}$ and if $v$ is adjacent
to $f$, then $G|bcdefv\iso \co{C_6}$.  Therefore $A_{bc}$ is
anticomplete to $\{a,f\}$, and from symmetry, so is $A_{de}$.

Now suppose $v \in A_{bd}$. If $v$ is adjacent to $a$, then $G|acdev\iso
C_5$ and if $v$ is adjacent to $f$, then $G|cdefv\iso C_5$.  Therefore
$A_{bd}$ is anticomplete to $\{a,f\}$, and from symmetry, so is
$A_{ce}$.  It follows that $A_2$ is anticomplete to $\{a,f\}$.  Note
that for $v \in A_{bc} \cup A_{de}$, either $G|abcdev\iso \domino$ or
$G|bcdefv\iso \domino$, and so by an argument analogous to the one
above, we conclude that $A_{bc} \cup A_{de}$ is anticomplete to $A_{bd}
\cup A_{ce}$ and that $A_{bc} \cup A_{de}$ is a stable set; hence
$A_{bc} \cup A_{de}\cup \{a,f\}$ is a stable set.

It remains to show that $A_{bd} \cup
A_{ce}$ is a stable set.  For suppose $u,v \in A_{bd}$ are adjacent; then
$G|bcdeuv\iso \co{\watch}$.  Therefore $A_{bd}$ is a stable set and from
symmetry, so is $A_{ce}$.  Next, suppose $u \in A_{bd}$ and $v \in A_{ce}$ are
adjacent;  then $G|bcdeuv\iso\co{C_6}$.  Therefore $A_2 \cup \{a,f\}$ is
a stable set and this proves (2).
\betweenspace
Now $\{b,c,d,e\}$ is anti-matched by definition and $A_2 \cup \{a,f\}$
is a stable set by (2).  It remains to show that $A_2 \cup \{a,f\}$ and
$\{b,c,d,e\}$ are aligned.  Since $A_2 \cup \{a,f\}$ is a stable set, it
suffices to show that for all $v \in A_2 \cup \{a,f\}$, $v$ is adjacent
to exactly one of $b,e$ and exactly one of $c,d$.  For $v \in \{a,f\}$
this is true by definition, and for $v \in A_2$ this follows from (2).
Therefore $G$ is doubled and this proves \ref{Domino}.
\end{proof}

\begin{theorem}\label{Tent1}
A graph containing $\tent_1$ but no graphs in $\mathcal{F}$ is doubled.
\end{theorem}
\begin{proof}
Let $G$ be a graph containing $\tent_1$ but no graphs in $\mathcal{F}$.
By \ref{M_{2,1}}, \ref{P_5}, \ref{C_6}, and \ref{Domino}, we may assume
that $G$ does not contain $M_{2,1}$, $P_5$, $C_6$, $\domino$ or their
complements as induced subgraphs.  Let $G|abcdef\iso\tent_1$, where
$ab$, $bc$, $cd$, $de$, $fa$, $fb$, $fc$, and $fe$ are the edges; let $S =
\{a,b,c,d,e,f\}$.  For $0\leq i \leq 4$, let $A_i \subseteq V(G)
\setminus S$ denote the set of vertices that have $i$ neighbors in
$\{b,c,d,e\}$. Our goal is to show the following:
\begin{itemize}
\item $A_0 = A_2 = A_4 = \emptyset$, and
\item $G|(A_1 \cup \{a,c,d\})$ is semi-matched, and
\item $G|(A_3 \cup \{b,e,f\})$ is semi-antimatched, and
\item $A_1 \cup \{a,c,d\}$ and $A_3 \cup \{b,e,f\}$ are aligned.
\end{itemize}
Together, these statements imply that $G$ is doubled.

\step{(1) $A_0 = A_2 = A_4 = \emptyset$.}

Suppose $v \in A_0$.  If $v$ is adjacent to $a$, then $G|abcdev\iso
P_5$, and if $v$ is not adjacent to $a$, then $G|abdev\iso M_{2,1}$.
Therefore $A_0 = \emptyset$.

Next, suppose $v \in A_4$.  If $v$ is adjacent to $f$, then
$G|cdefv\iso\co{M_{2,1}}$, and if $v$ is not adjacent to $f$, then
$G|bcdefv\iso \co{P_5}$.  Therefore $A_4 = \emptyset$.

Next, we show that $A_2=\emptyset$.  For $u,v \in \{b,c,d,e\}$, let
$A_{uv}\subseteq A_2$ be those vertices that are adjacent to $u$ and
$v$.  If $v \in A_{be}$, then $G|bcdev\iso C_5$, and so $A_{be} =
\emptyset$.  Now suppose $v \in A_{bc}$.  If $v$ is adjacent to $a$,
then $G|abdev\iso\co{K_{2,3}}$ and if $v$ is adjacent to $f$, then
$G|bcdefv\iso \co{\watch}$.  So we may assume that $v$ is not adjacent
to either $a$ or $f$; but then $G|S \cup \{v\}\iso F_{20}$.  Therefore
$A_{bc} = \emptyset$.

Next, suppose $v \in A_{bd}$.  If $v$ is not adjacent to $f$, then
$G|bdefv\iso C_5$, and if $v$ is not adjacent to $a$, then $G|abcdev\iso
\watch$.  Hence, we may assume that $v$ is adjacent to both $a$ and $f$;
but then $G|abdefv\iso \co{\watch}$.  Therefore $A_{bd} = \emptyset$.

Next, suppose $v \in A_{cd}$.  If $v$ is adjacent to both $a$ and $f$, then
$G|abcfv\iso \co{M_{2,1}}$.  Next, if $v$ is adjacent to $a$ but not to $f$,
then $G|abcdfv\iso \co{\flag}$, and if $v$ is adjacent to $f$ but not to
$a$, then $G|abdefv\iso \fish$.  So we may assume that $v$ is not adjacent to
$a$ or $f$; but then $G|abdfv\iso \co{K_{2,3}}$.  Therefore $A_{cd} =
\emptyset$.

Next, suppose $v \in A_{ce}$.  Then $G|abcdev\iso \domino$ or
$\flag$ depending on the adjacency between $v$ and $a$.  Therefore
$A_{ce} = \emptyset$.

So we may assume that $v \in A_{de}$.  If $v$ is adjacent to $a$, then
$G|abcdv\iso C_5$, and if $v$ is not adjacent to $a$, then $G|abdev\iso
\co{K_{2,3}}$.  Therefore $A_2 = \emptyset$ and this proves (1).

\step{(2) $A_1$ is complete to $b$.}

For $u \in \{b,c,d,e\}$, let $A_{u}\subseteq A_1$ be those vertices that
are adjacent to $u$.  We will show that $A_c=A_d=A_e=\emptyset$.

Suppose $v \in A_c$.  If $v$ is adjacent to $a$, then $G|abcdev\iso
\flag$, and if $v$ is not adjacent to $a$, then $G|abdev\iso M_{2,1}$.
Therefore $A_c = \emptyset$.

Next, suppose $v \in A_e$.  Then $G|abcdev\iso P_5$ or $C_6$
depending on the adjacency between $v$ and $a$.  Therefore $A_e =
\emptyset$.

Next, suppose $v \in A_d$.  If $v$ is adjacent to $a$, then $G|abcdv\iso
C_5$, and if $v$ is adjacent to $f$ but not to $a$, then $G|abdefv\iso
\fish$.  So we may assume that $v$ is not adjacent to either $a$ or $f$;
but then $G|acdefv\iso \watch$.  Therefore $A_d = \emptyset$.  This
completes that proof of (2).

\step{(3) $A_1 \cup \{a\}$ is a stable set.}

Suppose $v \in A_1$ and $a$ are adjacent; then $G|abdev\iso
\co{K_{2,3}}$.  Therefore $A_1$ is anticomplete to $a$.  Next, suppose $u,v
\in A_1$ are adjacent.  Then $G|bdeuv\iso\co{K_{2,3}}$.
Therefore $A_1 \cup \{a\}$ is a stable set and this proves (3).

\step{(4) If $v \in A_3$, then $v \in A_{bce} \cup A_{bde}$.}

For $u,v,w \in \{b,c,d,e\}$, let $A_{uvw}\subseteq A_3$ be those
vertices that are adjacent to $u,v$ and $w$.

Suppose $v \in A_{bcd}$.  If $v$ is not adjacent to $f$, then
$G|bdefv\iso C_5$, and if $v$ is adjacent to $f$ but not to $a$, then
$G|S \cup \{v\}\iso \co{F_{21}}$.  So we may assume that $v$ is adjacent
to both $a$ and $f$; but then $G|abdefv\iso \co{\watch}$.  Therefore
$A_{bcd} = \emptyset$.

Next, suppose $v \in A_{cde}$.  If $v$ is adjacent to $f$, then
$G|cdefv\iso\co{M_{2,1}}$, and if $v$ is not adjacent to $f$, then
$G|bcdefv\iso \co{\flag}$.  Therefore $A_{cde} = \emptyset$.  This
proves (4).

\step{(5) $A_3 \cup \{f\}$ is a clique.}

Suppose $v \in A_{bde}$ and $v$ is not adjacent to $f$.  Then
$G|bcdefv\iso\co{C_6}$.  Next, suppose $v \in A_{bce}$ not adjacent to $f$.
Then $G|cdefv\iso K_{2,3}$.  Therefore $A_3$ is complete to $f$.

Next, suppose $u,v \in A_{bde} \cup A_{bce}$ are not adjacent.  If $u,v
\in A_{bde}$, then $G|bcduv\iso K_{2,3}$ and if $u,v \in A_{bce}$ then
$G|cdeuv\iso K_{2,3}$.  So we may assume that $u \in A_{bde}$ and $v \in
A_{bce}$; but then $G|bcdeuv\iso\co{C_6}$.  Therefore $A_3 \cup \{f\}$
is a clique and this proves (5).
\betweenspace
From (2) and (3), it follows that $A_1 \cup \{a,c,d\}$ is semi-matched
with one edge (namely, $cd$).  From (4) and (5), $A_3 \cup \{b,e,f\}$ is
semi-antimatched with one nonedge (namely, $be$).  Furthermore, it
follows by definition and from (2) that for all $u \in A_1 \cup
\{a,c,d\}$, $u$ is adjacent to exactly one of $b$ and $e$.  It also
follows by definition and from (4) that for all $v \in A_3 \cup
\{b,e,f\}$, $v$ is adjacent to exactly one of $c$ and $d$.  Therefore
$A_1 \cup \{a,c,d\}$ and $A_3 \cup \{b,e,f\}$ are aligned and this
proves \ref{Tent1}.
\end{proof}

\begin{theorem}\label{Tent2}
A graph containing $\tent_2$ but no graphs in $\mathcal{F}$ is doubled.
\end{theorem}
\begin{proof}
Let $G$ be a graph containing $\tent_2$ but no graphs in $\mathcal{F}$.
By \ref{M_{2,1}}, \ref{P_5}, \ref{C_6}, \ref{Domino}, and \ref{Tent1},
we may assume that $G$ does not contain $M_{2,1}$, $P_5$, $C_6$,
$\domino$, $\tent_1$ or their complements as induced subgraphs.  Let
$G|abcdef\iso\tent_2$, where $ab$, $bc$, $cd$, $de$, $fa$, $fb$, $fd$,
and $fe$ are the edges; let $S = \{a,b,c,d,e,f\}$.  First, we show that
if $v \in V(G) \setminus S$, then $N_S(v)$ is equal to $\{b,f\}$,
$\{d,f\}$, or $\{a,b,d,e,f\}$.

Let $A_{bf}$ be those vertices whose neighbor set in $S$ is $\{b,f\}$
and define $A_{df}$ and $A_{abdef}$ similarly.  We also prove that at
least one of $A_{bf}, A_{df}$ and $A_{abdef}$ is empty.  Then our goal
is to show the following:

If $A_{abdef} = \emptyset$, then
\begin{itemize}
\item $G|(A_2 \cup \{a,e\})$ is semi-matched, and
\item $G|\{b,c,d,f\}$ is antimatched, and
\item $A_2 \cup \{a,e\}$ and $\{b,c,d,f\}$ are aligned.
\end{itemize}

If $A_{df} = \emptyset$, then
\begin{itemize}
\item $G|(A_2 \cup \{a,c,d\})$ is semi-matched, and
\item $G|(A_5 \cup \{b,e,f\})$ is semi-antimatched, and
\item $A_2 \cup \{a,c,d\}$ and $A_5 \cup \{b,e,f\}$ are aligned.
\end{itemize}

If $A_{bf} = \emptyset$, then
\begin{itemize}
\item $G|(A_2 \cup \{b,c,e\})$ is semi-matched, and
\item $G|(A_5 \cup \{a,d,f\})$ is semi-antimatched, and
\item $A_2 \cup \{b,c,e\}$ and $A_5 \cup \{a,d,f\}$ are aligned.
\end{itemize}
Together, these statements imply that $G$ is doubled.

\step{(1) For $v \in V(G) \setminus S$, $N_S(v)$ is equal to $\{b,f\}$,
  $\{d,f\}$, or $\{a,b,d,e,f\}$.}

We show that $N_{\{b,c,d\}}(v)$ is equal to $\{b\}$, $\{d\}$, or
$\{b,d\}$ and for each case, $N_S(v)$ is equal to $\{b,f\}$, $\{d,f\}$,
or $\{a,b,d,e,f\}$, respectively.

First, suppose $N_{\{b,c,d\}}(v) = \emptyset$.  If $v$ is complete to
$\{a,e\}$, then $G|abcdev\iso C_6$, and if $v$ is adjacent to exactly one
of $a$ and $e$, then $G|abcdev\iso P_5$.  So we may assume that $v$ is
anticomplete to $\{a,e\}$; but then $G|abdev\iso M_{2,1}$.  Therefore
$N_{\{b,c,d\}}(v)$ cannot be empty.

Next, suppose $N_{\{b,c,d\}}(v) = \{b\}$.  If $v$ is adjacent to $e$, then
$G|bcdev\iso C_5$, and if $v$ is adjacent to $a$, then
$G|abdev\iso \co{K_{2,3}}$.  If $v$ is not adjacent to
$f$, then $G|bcdefv\iso \co{\tent_1}$.  Therefore $N_S(v)
= \{b,f\}$.  Similarly if $N_{\{b,c,d\}}(v) = \{d\}$, then $N_S(v) =
\{d,f\}$.

Next, suppose $N_{\{b,c,d\}}(v) = \{c\}$.  If $v$ is complete to
$\{a,e\}$, then $G|abcdev\iso \domino$, and if $v$ is adjacent to
exactly one of $a$ and $e$, then $G|abcdev\iso \flag$.  So we may assume
that $v$ is anticomplete to $\{a,e\}$; but then $G|abdev\iso M_{2,1}$.
Therefore $N_{\{b,c,d\}}(v)$ cannot be equal to $\{c\}$.

Next, suppose $N_{\{b,c,d\}}(v) = \{b,c\}$.  If $v$ is complete to
$\{e,f\}$, then $G|bcdefv\iso \co{P_5}$.  If $v$ is adjacent to $e$ but
not to $f$, then $G|bcdefv\iso \co{C_6}$, and if $v$ is adjacent to $f$
but not to $e$, then $G|bcdefv\iso\co{\flag}$.  So we may assume that
$v$ is anticomplete to $\{e,f\}$; but then $G|bcdefv\iso\co{\domino}$.
Therefore $N_{\{b,c,d\}}(v)$ cannot be $\{b,c\}$ and from symmetry,
$N_{\{b,c,d\}}(v)$ cannot be $\{c,d\}$.

Next, suppose $N_{\{b,c,d\}}(v) = \{b,d\}$.  If $v$ is not adjacent to
$f$, then $G|bcdfv\iso K_{2,3}$.  If $v$ is anticomplete to $\{a,e\}$,
then $G|abcdev\iso \watch$.  If $v$ is adjacent to one of $a$ and $e$,
then $G|abcdev\iso\co{\tent_1}$.  Therefore $N_S(v) = \{a,b,d,e,f\}$.

Finally, suppose $N_{\{b,c,d\}}(v) = \{b,c,d\}$.  If $v$ is adjacent to
$f$, then $G|bcdfv\iso \co{M_{2,1}}$.  If $v$ is not adjacent to $a$,
then $G|abcdfv\iso \co{\flag}$, while if $v$ is adjacent to $a$, then
$G|abcdfv\iso \co{P_5}$.  Therefore $v$ cannot be complete to
$\{b,c,d\}$.

Together, these statements prove (1).
\betweenspace
\step{(2) $A_{bf} \cup A_{df}$ is a stable set, and $A_{abdef}$ is a clique
complete to $A_{bf} \cup A_{df}$.}

Suppose $u,v \in A_{bf}$ are adjacent; then $G|bcdfuv\iso\co{\watch}$.
Therefore $A_{bf}$ is a stable set and similarly, so is $A_{df}$.
Now suppose $u \in A_{bf}$ and $v \in A_{df}$ are adjacent.  Then
$G|bcduv\iso C_5$.  Therefore $A_{bf} \cup A_{df}$ is a stable set.

Next, suppose $u,v \in A_{abdef}$ are not adjacent; then $G|bcduv\iso
K_{2,3}$.  Therefore $A_{abdef}$ is a clique.

Finally, suppose $u \in A_{bf}$ and $v \in A_{abdef}$ are not adjacent.
Then $G|abcdev \iso\tent_2$ and $u$ has only one neighbor in
$\{a,b,c,d,e,v\}$, which is impossible by (1).  Therefore $A_{abdef}$ is
complete to $A_{bf}$ and similarly to $A_{df}$, and this proves (2).

\step{(3) At least one of $A_{bf}, A_{df}$, and $A_{abdef}$ is empty.}

Suppose $u \in A_{bf}$, $v \in A_{df}$, and $w \in A_{abdef}$.  From
(2), $w$ is complete to $\{u,v\}$ and $u$ is not adjacent to $v$.  It follows
that $G|abcdefuvw\iso \co{F_{22}}$.  Therefore at least one of
$A_{bf}, A_{df}$, and $A_{abdef}$ is empty and this proves (3).
\betweenspace
If $A_{abdef}=\emptyset$, then it follows from (2) that $G|(A_{bf} \cup
A_{df} \cup \{a,e\})$ is a stable set.  Also, $G|bcdf$ is antimatched by
assumption and $A_{bf} \cup A_{df} \cup \{a,e\}$ and $\{b,c,d,f\}$ are
aligned by assumption and definition.  Hence, $G$ is doubled.

So we may assume that $A_{abdef} \neq \emptyset$.  Then by (3), one of
$A_{bf}$ and $A_{df}$ is empty and from symmetry, we may assume $A_{df}$
is empty.  Then $G|(A_{bf} \cup \{a,c,d\})$ is semi-matched with an edge
$cd$, and $G|(A_{abdef} \cup \{b,e,f\})$ is semi-antimatched with a
non-edge $be$.  It also follows from assumption and definition that for
all $u \in A_{bf} \cup \{a,c,d\}$, $u$ is adjacent to exactly one of $b$
and $e$ and for all $v \in A_{abdef} \cup \{b,e,f\}$, $v$ is adjacent to
exactly one of $c$ and $d$.  Hence, $G$ is doubled.  This proves
\ref{Tent2}.
\end{proof}

We are now ready to prove the main result.

\begin{proof}[Proof of \ref{dsplit}.]
  The ``only if'' part is obvious since none of the graphs in
  $\mathcal{F}$ are doubled. For the ``if'' part, we may assume $G$ is
  not almost-split and hence $G$ or $\co{G}$ contains one of $M_{2,1}$,
  $P_5$, $C_6$, $\domino$, $\tent_1$, and $\tent_2$ as an induced
  subgraph.  But then we are done by \ref{M_{2,1}}, \ref{P_5},
  \ref{C_6}, \ref{Domino}, \ref{Tent1}, or \ref{Tent2} applied to $G$ or
  $\co{G}$, keeping in mind that the complement of a doubled graph is
  doubled.
\end{proof}

\section*{Acknowledgments}

The first author acknowledges the support of the NSF Graduate Research
Fellowship.  All of the authors would like to thank BIRS for organizing
the workshop ``New trends on structural graph theory'', where this
research was initiated.

\end{document}